\documentclass[12pt]{amsart}

\usepackage{amssymb}
\usepackage[T1]{fontenc}
\usepackage[pdftex,colorlinks=true]{hyperref}

\usepackage{enumerate}

\makeatletter
\@namedef{subjclassname@2010}{%
  \textup{2010} Mathematics Subject Classification}
\makeatother

\newtheorem{thm}{Theorem}[section]
\newtheorem{cor}[thm]{Corollary}
\newtheorem{lem}[thm]{Lemma}
\newtheorem{prop}[thm]{Proposition}

\theoremstyle{definition}

\newtheorem{rem}[thm]{Remark}

\numberwithin{equation}{section}

\newcommand{\wt}{\widetilde}
\newcommand{\R}{\mathbb{R}}
\newcommand{\Z}{\mathbb{Z}}
\providecommand{\abs}[1]{\left\lvert#1\right\rvert}

\DeclareMathOperator{\sgn}{sgn}
\DeclareMathOperator{\Map}{Map}

\DeclareMathOperator{\Loc}{Loc}
\DeclareMathOperator{\Prop}{Prop}
\DeclareMathOperator{\id}{id}
\DeclareMathOperator{\Iso}{Iso}
\DeclareMathOperator{\rank}{rank}
\DeclareMathOperator{\I}{I}
\DeclareMathOperator{\Gl}{Gl}

\begin{document}


\baselineskip=17pt



\title[The Hopf type theorem for equivariant local maps]{The Hopf type theorem for equivariant local maps} 

\author[P. Bart{\l}omiejczyk]
{Piotr Bart{\l}omiejczyk}
\address{Faculty of Applied Physics and Mathematics,
Gda{\'n}sk University of Technology,
Gabriela Narutowicza 11/12,
80-233 Gda{\'{n}}sk, Poland}
\email{pbartlomiejczyk@mif.pg.gda.pl}

\date{\today}

\begin{abstract}
We study otopy classes of equivariant local maps and
prove the Hopf type theorem for such maps
in the case of a real finite dimensional orthogonal
representation of a compact Lie group. 
\end{abstract}

\maketitle


\section*{Introduction} 
\label{sec:intro}
The famous Hopf theorem states that
if $M$ is an $n$-dimensional connected orientable closed
manifold, $S^n$ the $n$-dimensional sphere and
$[M, S^n]$ denotes the set of homotopy classes of
continuous maps from $M$ to $S^n$,
then the Brouwer degree map $\deg\colon [M, S^n]\to\Z$
is a bijection for $n\ge1$. 

Our purpose is to prove 
a natural equivariant version of this theorem.
More precisely, in our version we replace 
the set of homotopy classes by the set of otopy classes
and we take into account the group action.
Namely, let $V$ be a real finite dimensional orthogonal
representation of a compact Lie group $G$,
$\Omega$ be an open invariant subset of $V$ and
$\mathcal{F}_G[\Omega]$ denote
the set of equivariant otopy classes of
equivariant local maps with 
open invariant domains contained in $\Omega$.
Assume that $0\not\in\Omega$ or $\dim V^{G}>0$. 
Our main result states that 
there is a natural bijection
\[
\mathcal{F}_G[\Omega]\approx
\prod_{(H)}\Biggl(\sum_{i=1}^{n(H)}\mathbb{Z}\Biggr),
\]
where the product (finite) extends over
the orbit types $(H)$ in $\Omega$ satisfying 
$\dim WH = 0$, where $WH$ is the Weyl group of $H$,
and the direct sums (finite or countably infinite)
are taken over the sets of all connected components
of the quotients $\Omega_H/WH$.

The proof of our main theorem follows from the
following two results:
the splitting result which gives
a product decomposition of  $\mathcal{F}_G[\Omega]$
with respect to the orbit types $(H)$ in $\Omega$ satisfying 
$\dim WH = 0$ (proved in \cite{B} and \cite{GI})
and the full description of each factor
of the above product decomposition (proved here).
In turn, the proof of the second above result 
is based on the classification of fiber otopy classes of 
local cross sections of a vector bundle and 
establishing (in case of free actions) 
a bijection between the set of otopy classes 
of equivariant local maps  and the set of fiber otopy classes of 
local cross sections of some vector bundle.

Our paper contains also some remarks
concerning the parametrized case.
Although the consideration of the parameters
may seem to be an artificial generalization,
but, in fact, it is a natural way to study the higher
homotopy groups of the space $\mathcal{F}_G(\Omega)$
using the formula
$
\mathcal{F}_G[\R^k\times\Omega]
\approx\pi_{k}\left(\mathcal{F}_G\left(\Omega\right)\right)
$
established in \cite{B}.

The organization of the paper is as follows.
Section~\ref{sec:prel} presents some preliminaries.
In Section~\ref{sec:cross} we introduce the space of
local cross sections of a vector bundle and state
two results concerning the classification of
fiber otopy classes of such sections. 
Section~\ref{sec:proof} contains the proof
of the second of these results.
In Section~\ref{sec:free} we give the description of 
the set of otopy classes of equivariant local maps
in case of free actions.
Finally, in Section~\ref{sec:hopf} we prove
our main result i.e.\  the Hopf type theorem for equivariant local maps.


\section{Preliminaries} 
\label{sec:prel}
The notation $A\Subset B$ means 
that $A$ is a compact subset of $B$.
For a topological space $X$,
let $\tau(X)$ denote the topology on $X$.
Recall that if $A$, $B$ are topological spaces,
then $\Map(A,B)$ denotes the set of all continuous maps
of $A$ into $B$ equipped with the usual compact-open topology
i.e.\ having as subbasis all the sets
$\Gamma(C,U)=\{\,f\in\Map(A,B)\mid f(C)\subset U\,\}$
for $C\Subset A$ and $U$ open in $B$. 

For any topological spaces $X$ and $Y$, 
let $\mathcal{M}(X,Y)$ be the set
of all continuous maps $f\colon D_f\to Y$ such that 
$D_f$ is an open subset of $X$.
Let $\mathcal{R}$ be
a family of subsets of $Y$. We define 
\[
\Loc(X,Y,\mathcal{R}):=\{\,f\in\mathcal{M}(X,Y)\mid 
f^{-1}(R)\Subset D_f \text{ for all $R\in\mathcal{R}$}\,\}.
\]
We introduce a topology in $\Loc(X,Y,\mathcal{R})$
generated by the subbasis consisting of all sets of the form
\begin{itemize}
	\item $H(C,U):=\{\,f\in\Loc(X,Y,\mathcal{R})\mid 
	      C\subset D_f,\, f(C)\subset U\,\}$
	      for $C\Subset X$ and $U\in\tau(Y)$,
	\item $M(V,R):=\{\,f\in\Loc(X,Y,\mathcal{R})\mid 
	      f^{-1}(R)\subset V\,\}$ for $V\in\tau(X)$ and 
	      $R\in\mathcal{R}$.
\end{itemize}
Elements of $\Loc(X,Y,\mathcal{R})$ 
are called \emph{local maps}.
The natural base point of $\Loc(X,Y,\mathcal{R})$ 
is the empty map. 
The set-theoretic union of two local maps $f$ and $g$ 
with disjoint domains will be denoted by $f\sqcup g$.
We define the space of \emph{proper} maps $\Prop(X,Y)$
to be $\Loc(X,Y,\mathcal{K})$,
where $\mathcal{K}=\{K\mid K\Subset Y\}$.
Moreover, in the case when $\mathcal{R}=\{\{y\}\}$
we will write $\Loc(X,Y,y)$ omitting double curly brackets.

Assume that $G$ is a topological group
and $X$, $Y$ are two $G$-spaces.
We will denote by $\Map_G(X,Y)$ 
the subset of $\Map(X,Y)$
consisting of all maps $f$ such that
\[
f(gx)=gf(x)\quad
\text{for all $x\in X$ and $g\in G$}
\]
endowed with the relative topology.
Elements of  $\Map_G(X,Y)$ are called
\emph{equivariant maps} or \emph{$G$-maps}.
Let $\Loc_G(X,Y,\mathcal{R})$ (resp. $\Prop_G(X,Y)$)
be the subspace of $\Loc(X,Y,\mathcal{R})$ 
(resp. $\Prop(X,Y)$) consisting of equivariant maps
with invariant domains
and equipped with the induced topology.
Elements of $\Loc_G(X,Y,\mathcal{R})$
will be called \emph{equivariant local maps}
or \emph{local $G$-maps}.

Assume that $V$ is a real finite dimensional orthogonal
representation of a compact Lie group $G$.
Throughout the paper $\R^k$ denotes a trivial
representation of $G$. Let $\Omega$ be 
an open invariant subset of $\R^k\oplus V$.
Let us introduce the following notation:
\begin{align*}
\mathcal{F}_G(\Omega)&
:=\Loc_G(\Omega,V,0),\\ 
\mathcal{P}_G(\Omega)&:=\Prop_G(\Omega,V).
\end{align*}

Let $I=[0,1]$. 
We assume that the action
of $G$ on $I$ is trivial.
Any element of 
$\Loc_G(I\times\Omega,V,0)$
is called an \emph{otopy} and any element of 
$\Prop_G(I\times\Omega,V)$ is called a \emph{proper otopy}.
In \cite{B} we proved that 
each otopy (resp. proper otopy)
corresponds to a path in $\mathcal{F}_G(\Omega)$ 
(resp. $\mathcal{P}_G(\Omega)$) and vice versa.

Given a (proper) otopy 
$h\colon\Lambda\subset I\times\Omega\to V$ 
we can define for each $t\in I$
sets $\Lambda_t=\{x\in\Omega\mid(t,x)\in\Lambda\}$ and maps
$h_t\colon\Lambda_t\to V$ with $h_t(x)=h(t,x)$.
Note that from the above $h_t$ may be the empty map.
If $h$ is a (proper) otopy, we say that 
$h_0$ and $h_1$ are \emph{(proper) otopic}.
Of course, (proper) otopy gives an equivalence relation 
on $\mathcal{F}_G(\Omega)$ ($\mathcal{P}_G(\Omega)$).
The set of (proper) otopy classes will be denoted by 
$\mathcal{F}_G[\Omega]$ ($\mathcal{P}_G[\Omega]$).
In \cite[Prop. 2.4]{B} we showed that
the map $\mathcal{P}_G[\Omega]\to\mathcal{F}_G[\Omega]$
induced by the inclusion is a bijection.
Observe that if $f\in\mathcal{F}_G(\Omega)$ and $V$ is 
an open invariant subset of $D_f$ 
such that $f^{-1}(0)\subset V$, 
then $f$ and $f\vert_V$ are otopic.
In particular, if $f^{-1}(0)=\emptyset$ then $f$ 
is otopic to the empty map.

Recall that every Lie group is orientable.
Let $G$ be a compact Lie group.
For any $g\in G$ we will denote by $A_g\colon G\to G$
the automorphism given by $A_g(h)=ghg^{-1}$ for all $h\in G$.
Let $T_eG$ denote the tangent space at the unit element.
We say that $G$ is \emph{biorientable} if for each $g\in G$
the derivative $DA_g(e)\colon T_eG\to T_eG$ preserves
the orientation. 
It is not hard to show that if $G$ is abelian, finite or has an odd number
of connected components, then $G$ is biorientable.
The group $O(2)$ seems to be the simplest example
of a compact Lie group which is not biorientable.
Let $V$ be a real finite dimensional representation of a group $G$.
We say that the action of $G$ is \emph{orientation-preserving}
if for each  $g\in G$ the map $v\mapsto g\cdot v$ preserves
the orientation of $V$. In Section \ref{sec:free} we will need
the following result, which can be found in \cite[Lem. 3.4]{GKW}.

\begin{lem}\label{lem:biorient}
Assume $V$ is a real finite dimensional orthogonal
representation of a compact Lie group $G$
and $U$ is an open invariant subset of $V$
on which $G$ acts freely.
If $G$ is biorientable and the action of $G$
is orientation-preserving then $U/G$ is orientable.
\end{lem}
\begin{rem}
If $G$ is not biorientable then $U/G$ may be
orientable or not, or even may consist of many
components, of which some are orientable
and some are not.
\end{rem}

Assume $V$ is a real finite dimensional orthogonal
representation of a compact Lie group $G$
and $H$ is a closed subgroup of $G$.
Recall that $G_x=\{g\in G\mid gx=x\}$,
$(H)$ stands for a conjugacy class of $H$
and $WH=NH/H$, where $NH$ is a normalizer of $H$ in $G$.
Let $\Omega$ be an open invariant subset of $V$.
We define the following subsets of $\Omega$:
\begin{align*}
\Omega^H &=\{x\in\Omega\mid H\subset G_x\},\\
\Omega_H &=\{x\in\Omega\mid H=G_x\},\\
\Omega_{(H)} &=\{x\in X\mid (H)=(G_x)\}.
\end{align*}
Let
\begin{align*} 
\Phi(G) &=\{(H)\mid \text{$H$ is a closed subgroup of $G$}\},\\
\Iso(\Omega) &=\{(H)\in\Phi(G)\mid\Omega_{(H)}\neq\emptyset\}.
\end{align*}
The set $\Iso(\Omega)$ is partially ordered.
Namely, $(H)\le(K)$ if $H$ is conjugate to a subgroup of $K$.

Throughout the paper we will make use of
the following well-known facts:
\begin{itemize}
  \item $WH$ is a compact Lie group,
	\item $V^H$ is a linear subspace of $V$ and orthogonal
	       representation of $WH$,
	\item the action of $WH$ on $\Omega_H$ is free,
	\item $\Omega_H$ is open and dense in $\Omega^H$.
\end{itemize}


\section{The space of local cross sections of a vector bundle} 
\label{sec:cross}

All manifolds considered are without boundary.
Assume $p\colon E\to M$ is a smooth (i.e., $C^1$) vector bundle. 
We will identify $M$ with the zero section of $E$.
A \emph{local cross section} of a bundle $p\colon E\to M$
is a continuous map $s\colon U\to E$, where $U$ is open in $M$,
$s^{-1}(M)$ is compact and $p\circ s=\id_U$. 
Let $\Gamma_{\text{loc}}(M,E)$ denote 
the set of all local cross sections of $E$ over $M$.
A \emph{fiber otopy} is a continuous map $h\colon\Lambda\to E$
such that $\Lambda$ is open in $I\times M$,
$h^{-1}(M)$ is compact and $p\left(h(t,x)\right)=x$
for all $(t,x)\in\Lambda$.
Let $s,s'\in\Gamma_{\text{loc}}(M,E)$.
We say that $s$ and $s'$ are \emph{fiber otopic}
provided there is a fiber otopy $h$ such that
$h_0=s$ and $h_1=s'$, where $h_t(x)=h(t,x)$.
Of course, fiber otopy gives an equivalence relation 
on $\Gamma_{\text{loc}}(M,E)$,
which will be denoted by $s\sim s'$.
Let $\Gamma_{\text{loc}}[M,E]$ denote 
the set of fiber otopy classes of 
local cross sections of $p\colon E\to M$.
Observe that if $s$ is a local cross section and $V$ is 
an open subset of $D_s$ such that $s^{-1}(M)\subset V$, 
then $s$ and $s\vert_V$ are fiber otopic.
This property of local cross section 
will be called \emph{localization}.
In particular, if $s^{-1}(M)=\emptyset$ then $s$ 
is fiber otopic to the empty map.
The set-theoretic union of two local cross sections $s$ and $s'$ 
with disjoint domains will be denoted by $s\sqcup s'$.

The following result has been proved in \cite[Prop. 3.2]{B}.

\begin{prop}\label{prop:sec}
If $\rank E>\dim M$ then $\Gamma_{\text{loc}}[M,E]$
has a single element.
\end{prop}

In the remainder of this and in the next section
we assume that $\rank E=\dim M$.
Let us denote by $\I_2(s)$ the mod $2$ intersection number
and by $\I(s)$ the oriented intersection number
of a local cross section $s$ with the zero section
(see for instance \cite{GP}).
It is evident that in both cases the intersection number
is otopy invariant i.e.\ if two local cross sections
are fiber otopic then they have the same intersection number,
but more interestingly, the converse is also true.
Namely, the following result, which may be viewed as a version
of the well-known Hopf theorem, will be proved in the next section.

\begin{thm}\label{thm:intersection}
Let $M$ be connected.
\begin{enumerate}
	\item If $E$ is orientable as a manifold then 
	      $\I\colon\Gamma_{\text{loc}}[M,E]\to\Z$ is a bijection.
	\item If $E$ is non-orientable as a manifold then 
	      $\I_2\colon\Gamma_{\text{loc}}[M,E]\to\Z_2$ is a bijection.				
\end{enumerate}
\end{thm}

\begin{rem}
Observe that if $E=TM$ then local cross sections are local vector fields.
Let $\mathcal{V}_{\text{loc}}[M]:=\Gamma_{\text{loc}}[M,TM]$. Since
$TM$ is always orientable as a manifold we obtain that
$\I\colon\mathcal{V}_{\text{loc}}[M]\to\Z$ is a bijection
for $M$ smooth connected.
\end{rem}

\begin{rem}\label{rem:rank}
For now, we do not have any satisfactory
description of $\Gamma_{\text{loc}}[M,E]$
if $\rank E<\dim M$.
\end{rem}

\section{Proof of Theorem \texorpdfstring{\ref{thm:intersection}}{2.2}} 
\label{sec:proof}

We have divided the proof of Theorem \ref{thm:intersection}
into a sequence of lemmas and propositions.
We will need the following transversality result 
for sections of a smooth vector bundle,
which can be easily derived from 
the transversality theorem for maps.
A smooth local cross section is called \emph{generic}
if it is transverse to the zero section.

\begin{lem}\label{lem:trans}
Arbitrarily close (in the $C^1$ sense) 
to any smooth local cross section 
of a smooth vector bundle there exists
a generic local cross section
(both have the same domain).
\end{lem}

Let us denote by $B(p;r)$ the open $r$-ball in $M$ around $p$.
A generic local cross section $s$ is called \emph{standard}
if $D_s=B(p;r)$ and $s^{-1}(M)=\{p\}$.
The proof of Theorem \ref{thm:intersection}
is based on the following propositions and lemmas.

\begin{prop}\label{prop:union}
Any local cross section is otopic to a finite disjoint union
of standard ones.
\end{prop}

\begin{proof}
By Lemma \ref{lem:trans}, any local cross section is otopic
(also homotopic) to generic one and by localization property,
any generic local cross section is otopic
to a finite disjoint union of standard ones.
\end{proof}

Let us state three lemmas concerning standard
local cross sections and needed below. 
We omit their proofs since they 
either follow directly from definitions
or are very similar to those from the topological degree theory
(see for example \cite{Br}). 
In the following lemmas we assume that all considered 
local cross sections are defined on the ball $B(p;r)$.
Moreover, assume that there are a local trivialization
$\psi\colon U\times\R^n\to E$ and a coordinate system
$\phi\colon W\subset M\to\R^n$ such that
$B(p;r)\subset U\cap W$ and $\phi(p)=0$.
Of course, if $r$ is small enough such a trivialization
and a coordinate system always exist.
Furthermore, under the above assumptions
any standard local cross section has the form
$s(x)=\psi\left(x,f\left(\phi(x)\right)\right)$
with $f\colon\phi(B(p;r))\subset\R^n\to\R^n$ smooth. 
Finally, observe that $s_M(x)=\psi(x,M\phi(x))$
is a standard local cross section for any $M\in\Gl_n(\R)$.

\begin{lem}\label{lem:inlinear}
Let $A\in\Gl_n(\R)$.
\begin{enumerate}
	\item If $E$ is orientable as a manifold then $\I(s_A)=\sgn\det A$.
	\item If $E$ is non-orientable as a manifold then $\I_2(s_A)=1$.				
\end{enumerate}
\end{lem}

\begin{lem}\label{lem:sgndet}
Let $A,B\in\Gl_n(\R)$. 
\begin{enumerate}
	\item If $E$ is orientable as a manifold and $\sgn\det A=\sgn\det B$
	      then $s_A$ and $s_B$ are fiber otopic.
	\item If $E$ is non-orientable as a manifold 
	      then $s_A$ and $s_B$ are fiber otopic.				
\end{enumerate}
\end{lem}

\begin{lem}\label{lem:linear}
If $s(x)=\psi\left(x,f\left(\phi(x)\right)\right)$
is a standard local cross section 
then so is $\wt{s}(x)=\psi(x,Df(0)\phi(x))$ and there is
$R\in(0,r)$ such that the straight-line homotopy
between $s$ and $\wt{s}$ on $B(p;R)$ is a fiber otopy.
\end{lem}

The above map $\wt{s}$ will be called a \emph{linearization}
of the standard local cross section~$s$.
The next result follows easily from Lemmas
\ref{lem:inlinear} and \ref{lem:linear}.

\begin{cor}\label{cor:normal}
Let $s$ be a standard local cross section.
\begin{enumerate}
	\item If $E$ is orientable as a manifold then $\abs{\I(s)}=1$.
	\item If $E$ is non-orientable as a manifold then $\I_2(s)=1$.				
\end{enumerate}
\end{cor}

\begin{prop}\label{prop:hopfstand}
Let $s$ and $s'$ be standard local cross sections.
\begin{enumerate}
	\item If $E$ is orientable as a manifold
	      and $\I(s)=\I(s')$ then $s\sim s'$.
	\item If $E$ is non-orientable as a manifold
	      then $s\sim s'$.				
\end{enumerate}
\end{prop}

\begin{proof}
\emph{Step 1}. 
Observe that any standard local
cross section can be ``translated'' from one point 
on the manifold $M$ to another
using a finite sequence of translations
in images of coordinate systems. Such a ``translation''
is not unique, but it is always fiber otopic to 
the original local cross section.
So without loss of generality we can assume
that $s$ and $s'$ are defined on the same ball.

\emph{Step 2}. 
By Lemma \ref{lem:linear} two standard local cross sections 
defined on the same ball
are fiber otopic to their linearizations and by Lemma
\ref{lem:sgndet} these linearizations are fiber otopic, 
which completes the proof.
\end{proof}

\begin{rem}\label{rem:hopfstand}
In fact,  in the orientable case,
we will need the following slight generalization
of Proposition \ref{prop:hopfstand}.
Consider two finite sequences
$\{s_i\}_1^n$ and $\{s'_i\}_1^n$
of standard local cross sections such that
for each sequence the domains of local cross sections 
are pairwise disjoint and all $2n$ sections have
the same intersection number.
Then $\sqcup_{i=1}^n s_i\sim\sqcup_{i=1}^n s'_i$.
To verify our claim, observe that,
the procedure described in Step 1 of the proof
of Proposition \ref{prop:hopfstand}
can be done simultaneously for all $s_i$
in such a way that their domains remain
pairwise disjoint during the ``translation''.
\end{rem}

\begin{prop}[Annihilation]\label{prop:annihil}
Assume that $s$ and $s'$ are standard local cross sections
with disjoint domains.
\begin{enumerate}
	\item If $E$ is orientable as a manifold
	      and $\I(s)=-\I(s')$ then $s\sqcup s'\sim\emptyset$.
	\item If $E$ is non-orientable as a manifold
	      then $s\sqcup s'\sim\emptyset$.				
\end{enumerate}
\end{prop}

\begin{proof}
As in Step 1 from the previous proof
we can ``move'' one of standard local cross sections
and, in consequence, we can assume 
that the domains of $s$ and $s'$ are contained
in the same trivialization $\psi$ and coordinate system $\phi$.
Let $D_s=B(p;r)$ and $D_{s'}=B(q;r)$.
By Lemma  \ref{lem:linear} $s$ and $s'$ are
fiber otopic to their linearizations and by Lemma
\ref{lem:sgndet} these linearizations are fiber otopic
to  $s_A(x)=\psi\left(x,A\left(\phi(x)-\phi(p)\right)\right)$ 
and  $s_B(x)=\psi\left(x,B\left(\phi(x)-\phi(q)\right)\right)$,
where $A$ is an identity and $B$ is a symmetry.
Finally, repeating the procedure from the proof
of Lemma 3.2 in \cite{BP1}
we obtain that $s_A$ and $s_B$ annihilate that is
$s_A\sqcup s_B\sim\emptyset$. 
\end{proof}

The following result is an easy consequence
of Propositions \ref{prop:union} and \ref{prop:annihil}.

\begin{prop}\label{prop:union2}
Let $s$ be a local cross section.
\begin{enumerate}
	\item If $E$ is orientable as a manifold and $\I(s)=m$ then 
	      $s$ is otopic to a disjoint union of $\abs{m}$
				standard local cross sections, each of them with the same
				intersection number equal to $1$ (resp.\ $-1$)
				if $m\ge0$ (resp.\ $m<0$).
	\item If $E$ is non-orientable as a manifold and $\I_2(s)=m$ then
	      $s$ is otopic to $m$ standard local cross sections.
\end{enumerate}					
\end{prop}

\begin{proof}[Proof of Theorem \ref{thm:intersection}]
Surjectivity follows easily from Lemma \ref{lem:inlinear}.
In turn, injectivity in the orientable case follows immediately
from Proposition \ref{prop:union2}(1) and Remark \ref{rem:hopfstand}
and in the non-orientable case from Proposition \ref{prop:union2}(2)
and Proposition \ref{prop:hopfstand}(2).
\end{proof}


\section{The set of otopy classes in case of free actions} 
\label{sec:free}

Assume that $V$ is a real finite dimensional orthogonal
representation of a compact Lie group $G$.
Let $\Omega$ be an open invariant subset of $V$,
$G$ act freely on $\Omega$,
$M:=\Omega/G$, $E:=(\Omega\times V)/G$.
The trivial vector bundle $\Omega\times V\to\Omega$ 
factorizes to the vector bundle $p\colon E\to M$
with the typical fiber $F=\R^n$.
In \cite{B} we proved the following result. 

\begin{thm}\label{thm:tomdieck}
The function 
$\Phi\colon\mathcal{F}_G(\Omega)\to
\Gamma_{\text{loc}}(M,E)$ given by
$\Phi(f):= s_f$, 
where $s_f\left([x]\right):=\left[x,f(x)\right]$,
is bijective. Moreover,
$\Phi$ induces a bijection between
$\mathcal{F}_G[\Omega]$ and $\Gamma_{\text{loc}}[M,E]$.
\end{thm}

\begin{rem}\label{rem:homeo}
Since $\Phi$ is a bijection, 
$\Gamma_{\text{loc}}(M,E)$ is given
the induced topology, in which
$\Phi$ is a homeomorphism. Moreover,
the topology on $\Gamma_{\text{loc}}(M,E)$
coincides with the relative topology induced
from $\Loc\left(M,E,\{M\}\right)$. Finally, both otopies
and fiber otopies correspond to paths
in the respective spaces and vice versa.
\end{rem}

Recall that if $\dim G>0$
then the set $\mathcal{F}_G[\Omega]$ 
has a single element (see \cite[Theorem~3.1]{B}). 
The next result is devoted to the case $\dim G=0$.
Note that the set
of connected components of the quotient $\Omega/G$
is either finite or countably infinite.
\begin{thm}\label{thm:zero}
If $G$ is zero-dimensional (i.e. finite) then
there is a natural bijection
\begin{equation}\label{eqn:zero}
\mathcal{F}_G[\Omega]\approx
\sum_\alpha\mathbb{Z},
\end{equation}
where the direct sum is taken
over the set of all connected components
$\alpha$ of the quotient $\Omega/G$.
\end{thm} 

\begin{proof}
If $\dim G=0$ then $\rank E=\dim M$.
Moreover, by Theorem \ref{lem:biorient},
$E$ is orientable as a manifold.
Now the assertion is an easy consequence
of Theorems \ref{thm:intersection} and
\ref{thm:tomdieck}.
Observe that a direct sum (not product)
appears in the right-hand side
of \eqref{eqn:zero},
since for any local cross section $s$
the preimage $s^{-1}(M)$ meets
only a finite number of components of $M$.
\end{proof}

We close this section with some remarks 
concerning the parametrized case.
Assume that $\Omega$ is 
an open invariant subset of $\R^k\oplus V$
with trivial action of $G$ on $\R^k$
and free action of $G$ on $\Omega$.
First, it is easy to see that
Theorem \ref{thm:tomdieck} is still true.
If $\dim G>k$ then $\mathcal{F}_G[\Omega]$
is trivial (see \cite[Theorem~3.6]{B})
and if $\dim G=k$ and $G$ is biorientable
then the formula \eqref{eqn:zero} holds (with similar proof).
However, if $\dim G=k$ and $G$ is not biorientable
then the right-hand side of \eqref{eqn:zero} is equal
to $\sum_\alpha A_\alpha$, where $A_\alpha$ is
either $\mathbb{Z}$ or $\mathbb{Z}_2$
depending on whether the respective component
$E_\alpha$ of $E$ is orientable as a manifold or not.
Finally, so far we are not able to give a similar description
of $\mathcal{F}_G[\Omega]$ if $\dim G<k$ 
according to Remark \ref{rem:rank}.


\section{The Hopf type theorem for equivariant local maps} 
\label{sec:hopf}
Assume that $V$ is a real finite dimensional orthogonal
representation of a compact Lie group $G$
and $\Omega$ is an open invariant 
subset of $\R^k\oplus V$.
Let
\[
\Phi_k(G)=\{(H)\in\Phi(G)\mid\dim WH\le k\}.
\]
It is well-known that the set 
$\Iso(\Omega)$ is finite and so is
$\Iso(\Omega)\cap\Phi_k(G)$.
The following splitting result was proved 
in \cite{B} and \cite{GI}.
\begin{thm}\label{thm:split}
There are bijections
\begin{align}
\mathcal{F}_G[\Omega]&\approx
\prod_{(H)}
\mathcal{F}_{WH}\left[\Omega_H\right],
\label{eqn:bij1}\\
\mathcal{P}_G[\Omega]&\approx
\prod_{(H)}
\mathcal{P}_{WH}\left[\Omega_H\right],
\label{eqn:bij2}
\end{align}
where the products are taken over 
the set $\Iso(\Omega)\cap\Phi_k(G)$.
\end{thm}

We can now formulate the main result of this paper,
which may be viewed as an equivariant version
of the well-known Hopf theorem.
Assume that $V$ is a real finite dimensional orthogonal
representation of a compact Lie group $G$ and
$\Omega$ is an open invariant subset of $V$.
\begin{thm}\label{thm:hopf}
If $0\not\in\Omega$ or $\dim V^{G}>0$, then
there is a natural bijection
\begin{equation}\label{eqn:hopf}
\mathcal{F}_G[\Omega]\approx
\prod_{(H)}\Biggl(\sum_{i=1}^{n(H)}\mathbb{Z}\Biggr),
\end{equation}
where the product is taken over 
the set $\Iso(\Omega)\cap\Phi_0(G)$
and the respective direct sums are indexed by
the either finite or countably infinite
sets of connected components
of the quotients $\Omega_H/WH$.
\end{thm} 

\begin{proof}
From \eqref{eqn:bij1} we obtain 
a product decomposition of  $\mathcal{F}_G[\Omega]$
with respect to the conjugacy classes
from $\Iso(\Omega)\cap\Phi_0(G)$ and Theorem \ref{thm:zero}
gives a full description of each factor
of this product decomposition,
which completes the proof.
\end{proof}

We end the paper with some remarks and open problems.

\begin{rem}
It should be emphasized that, apart from the otopy
invariance and the Hopf property, the bijection
from Theorem~\ref{thm:hopf} satisfies also
other properties of the topological degree,
which is an immediate consequence of our construction.
Let us formulate, as an example, the additivity 
and existence properties. Namely, if we denote the above
bijection by $\deg$, then
\begin{itemize}
	\item $\deg\left(\left[f\sqcup g\right]\right)=
	      \deg\left(\left[f\right]\right)+\deg\left(\left[g\right]\right)$
				for $f,g\in\mathcal{F}_G(\Omega)$ such that $D_f\cap D_g=\emptyset$,
	\item $\deg\left(\left[f\right]\right)\neq0$ implies $f(x)=0$ for some
	      $x\in\Omega$.
\end{itemize}
\end{rem}

\begin{rem}
Observe that the extreme case of the trivial action
covers the~classical Hopf theorem.
Namely, if $G$ acts trivially on $V$ and
$\Omega$ is an open connected subset of $V$,
then
\[
\mathcal{F}_G[\Omega]=
\mathcal{F}_{\{e\}}[\Omega]\approx\Z.
\]
\end{rem}

\begin{rem}
Of course, Theorem \ref{thm:hopf} remains true if we replace 
$\mathcal{F}_G[\Omega]$ by $\mathcal{P}_G[\Omega]$.
\end{rem} 

\begin{rem}
Although in the parametrized case the splitting formula
\eqref{eqn:bij1} still holds, but for now in this product splitting 
we can only describe factors for which $\dim WH=k$
as we mentioned at the end of the previous section.
\end{rem} 

\subsection*{Acknowledgements}
The author wishes to express his thanks to the referee
for helpful comments concerning the paper.

\end{document}